\documentclass{amsart}

\usepackage{url}

\usepackage{amscd}
\usepackage{amsfonts}
\usepackage{amssymb}
\usepackage{euscript}
\usepackage{amsmath}
\usepackage{amsthm}
\usepackage[matrix, arrow, curve]{xy}
\usepackage{mathrsfs}

\usepackage{dsfont}

\newcommand{\We}{\mathop{\mathcal{W}^n}}
\newcommand{\g}{\mathfrak{g}}

\newcommand{\Ch}{\mathop{\mathrm{Ch}^\bullet}}
\newcommand{\Chq}{\mathop{\mathrm{Ch}_h^\bullet}}

\newcommand{\obl}{\mathop\mathrm{obl}}
\renewcommand{\k}{\Bbbk}
\newcommand{\Tor}{\mathop{\mathrm{Tor}}}
\newcommand{\Ext}{\mathop{\mathrm{Ext}}}
\newcommand{\End}{\mathop{\mathrm{End}}}
\newcommand{\Ad}{\mathrm{Ad}}

\newcommand{\gd}{\g^\vee}
\newcommand{\Sgd}{\k[[\gd]]}
\newcommand{\F}{\mathcal{L}}
\newcommand{\id}{\mathrm{id}}
\newcommand{\Tr}{\mathop\mathrm{Tr}}
\newcommand{\tens}[1]{\mathbin{\mathop{\otimes}\limits_{#1}}}
\renewcommand{\j}{j}

\theoremstyle{plain}
\newtheorem{prop}{Proposition}

\newtheorem{cor}{Corollary}

\theoremstyle{definition}
\newtheorem{definition}{Definition}
\newtheorem{example}{Example}
\newtheorem*{construction}{The construction}
\newtheorem*{ackn}{Acknowledgments}

\theoremstyle{remark}
\newtheorem{remark}{Remark}

\title{Weyl $n$-algebras and the Kontsevich integral of the unknot}
\author{Nikita Markarian}

\date{}

\address{National Research University Higher School of Economics, Russian
 Federation,
Department of Mathematics, 20 Myasnitskaya str., 101000, Moscow,
Russia}

\email{nikita.markarian@gmail.com}

\thanks{I was partially supported by a subsidy granted to the HSE by the
 Government of the Russian Federation for the implementation
        of the Global Competitiveness Program and by   RFBR Grant NN.  15-01-09242.
}

\begin{document}

\begin{abstract}
Given a Lie algebra with a scalar product, one may consider the latter
as a symplectic structure on a $dg$-scheme, which is the spectrum
of the Chevalley--Eilenberg algebra. In the first section we explicitly 
calculate the first order deformation of the differential on the
Hochschild complex of the Chevalley--Eilenberg algebra. The answer 
contains the Duflo character.  This calculation is used in the last section.
There we sketch the definition
of the Wilson loop invariant of knots, which is, hopefully, equal to the Kontsevich integral,
and show that for unknot they coincide.
As a byproduct, we get a new proof of the Duflo isomorphism for
a Lie algebra with a scalar product.
\end{abstract}

\maketitle

\section*{Introduction}

In \cite{W} we built perturbative Chern--Simons invariants by means
of the factorization complex of Weyl $n$-algebras. In the present paper
we continue this line and introduce the Wilson loop invariant.
This invariant is supposed to be equal to the Bott--Taubes invariant 
and the Kontsevich integral. 
In fact, we are only interested in one question here:
calculating the Wilson loop invariant of unknot in $S^3$.
This problem appears to be connected with the Duflo isomorphism. 

We consider the  Duflo isomorphism for Lie algebras with a scalar
product, which is much simpler to prove than the general statement from \cite{Du}.
There are (at least) two proofs of the Duflo isomorphism for a Lie algebra
with a scalar product. In \cite{AM} the authors use
a quantization of the Weil algebra. In \cite{BN} the Kontsevich integral
of knots and link is used. 
Our sketch of a proof 
(see remark before Proposition \ref{superduflo}) is related to the both.
The work \cite{Kr} also connects these two approaches and it would be very interesting to
compare it with our arguments.
 
The first section is not strongly connected with the rest of the paper,
but is of independent interest.
Here we make a very concrete calculation of
the first order deformation  of the Hochschild complex 
for the Chevalley--Eilenberg algebra of a Lie algebra.
The deformation is given by the scalar product.
This calculation is closely connected with \cite{M1} and may be rephrased in the style of this paper,
see Remark \ref{remark}.

In the second section we give a very short survey of results about $e_n$-algebras
and the factorization complex we need. For basics we refer the reader to \cite{Lu}
and for a much more detailed survey than ours we refer to \cite{G}.
At the end of the section we describe a construction, which plays a crucial role in the next section.

In the third section we apply this construction to the quantum Chevalley--Eilenberg algebra,
the role of which for perturbative Chern--Simons invariants is explained in
\cite[Appendix]{W}. The central result here is Proposition \ref{superduflo}.
The calculation we make here strongly reminds the one from the first section.
I would like to understand better reasons of this  similarity.
This section must be considered as an announcement. It contains no proofs.

Everything is over a field $\k$ of characteristic $0$.

\begin{ackn} I am grateful to 
B.~Feigin,  O.~Gwilliam, D.~Kaledin, A.~Kalugin, M.~Karev and A.~Khoroshkin  for fruitful discussions. 
I warmly thank the referee for providing constructive comments and help in improving the contents of this paper.
I would  like to express my deepest gratitude to Sergei Duzhin for his kindness and sensibility.
\end{ackn}

\section{Quantization of  the Chevalley--Eilenberg complex}

\subsection{Hochschild homology of the Chevalley--Eilenberg complex}

Let $\g$ be a finite-dimensional Lie algebra. The Chevalley--Eilenberg algebra $\Ch(\g)$
is a super-commutative $dg$-algebra $S^*(\gd[1])$ generated by the dual space $\gd$ placed in degree $1$.
The differential is a derivation of this free super-commutative algebra defined on the generators
by the tensor $\gd \to \gd \wedge \gd $ dual to the bracket. The Jacobi identity guarantees that this is indeed,
a differential. In terms of \cite{AKSZ} the Chevalley--Eilenberg algebra may be thought of as
the function ring of a $Q$-manifold.

With any $\g$-module $E$ one may associate the module $\Ch(\g,E)$ over $\Ch(\g)$ as follows.
As a $S^*(\gd[1])$-module it is freely generated by $E$ and the differential
is defined by its value on $E\otimes 1$ given by the tensor
$E\to E\otimes\gd$ of the $\g$-action. As a complex, $\Ch(\g,E)$
calculates the cohomology of $\g$ with coefficients in $E$.

The $\Ch(\g)$-module  $\Ch(\g,\gd_{ad})$ corresponding to the adjoint $\g$-module
may be thought of as a cotangent complex of $\Ch(\g)$.
The de Rham differential $d_{dR}\colon \Ch(\g)\to \Ch(\g,\gd_{ad})$,
which is a derivation of $\Ch(\g)$-modules, is tautologically defined on the generators.  
Define the  $\Ch(\g)$-module of differential forms 
of $\Ch(\g)$ as $\Ch(\g,\Sgd^{ad})$. 
It is a super-commutative algebra and 
the de Rham differential acts on it in the usual way, it is a derivation.

For a  unital $dg$-algebra $A$ define
the reduced (or normalized) Hochschild complex 
$C_*(A)$ (see e.~g. \cite[Ch 1.1]{L}) as
the total complex of the bi-complex with the $(-i)$-th term
\begin{equation}
\prod_{i\ge 0} (A\otimes \underbrace{A/\k\otimes\cdots\otimes A/\k}_i),
\label{1}
\end{equation}
the first  differential coming from $A$ and the second differential given by
\begin{align}
 &a_0\otimes a_1\otimes a_2\otimes\cdots\otimes a_i
\mapsto\nonumber\\ 
&a_0 a_1\otimes a_2\otimes\cdots\otimes a_i-
a_0\otimes a_1 a_2\otimes\cdots\otimes a_i+\dots\nonumber\\
&+(-1)^{i+\deg a_i(\deg a_0+\dots \deg a_{i-1})}a_i a_0\otimes a_1\otimes\cdots\otimes a_{i-1}.
\label{2}
\end{align}
Here one have to choose representatives of quotients $A/\k$,
then apply formula and take quotients again,  the result
does not depend on choices.
Note, that the usual  definition uses direct sums instead of products,
but we need the one we gave. In other words, we shall consider unbounded chains, that is
the graded completion (\cite[Definition A.25]{CD}) of $\sum_{i\ge 0} (A\otimes \underbrace{A/\k\otimes\cdots\otimes A/\k}_i)$ with respect to the grading given by the grading  on $A$.
For an ungraded algebra the reduced Hochschild complex calculates $\Tor_*^{A\otimes A^o}(A, A)$.

The following proposition is a variant of the
Hochschild--Kostant--Rosenberg isomorphism. 
\begin{prop}
	The formula
	\begin{equation}
	a_0 \otimes a_1\otimes\cdots\otimes a_i
	\mapsto a_0\,d_{dR}a_1\cdots \,d_{dR}a_i
	\label{hkr}
	\end{equation}
	defines a morphism 
	from the reduced Hochschild complex $C_*(\Ch(\g))$ of the Chevalley--Eilenberg algebra 
	to its  
	differential forms $\Ch(\g,\Sgd^{ad})$. This  morphism is a quasi-iso\-mor\-phism.
	\label{HKR}
\end{prop}
\begin{proof}
	Direct calculation shows that this is a morphism.
	The proof of Proposition \ref{tor} implies that this is a 
	quasi-isomorphism.	
\end{proof}

Equip $C_*(\Ch(\g))$ with a descending filtration $F$: the subcomplex
$F_kC_*(\Ch(\g))$ is spanned by chains $a_0\otimes a_1\otimes\cdots\otimes a_i$
such that $\deg a_0\ge k$.

\begin{prop}
	The spectral sequence associated with the filtration $F$ on $C_*(\Ch(\g))$
	degenerates at  the second sheet. The complex $E_1^{p,0}$ is isomorphic
	to $\Ch(\g,\Sgd^{ad})$ and $E_1^{p,>0}=0$. 
\label{tor}
\end{prop}
\begin{proof}
The associated graded object to the filtration $F$ is the tensor product
of $S^*(\gd[1])$ and the normalized standard complex, which calculates the
homology of algebra $\Ch(\g)$ with coefficients in the augmentation module.
More precisely, the latter complex is the total complex of 
the bicomplex, which is the direct product
$\prod_i(\Ch(\g)/\k)^{\otimes i}$, and with the second differential defined on $a_1\otimes\cdots\otimes a_i$
by
\begin{equation}
a_1\cdot a_2\otimes\cdots\otimes a_i-a_1\otimes a_2\cdot a_3\otimes\cdots\otimes a_i+\dots\pm 
a_1\otimes\cdots\otimes a_{i-1}\cdot a_i,
\label{stc}
\end{equation}
where $a_i$ are elements of the augmentation ideal, which is identified
with $\Ch(\g)/\k$. To compute its cohomology consider
the spectral sequence associated with the above mentioned bicomplex with the first differential
(\ref{stc}). It degenerates at the first sheet for trivial reasons and equals
$\Sgd$ sitting in degree $0$.

Equip $\Ch(\g,\Sgd^{ad})$ with the stupid filtration (e.~g. \cite[III.7.5]{GM})
and
consider  the map (\ref{hkr}) of filtered 
complexes. In the light of the above, the associated map of spectral sequences
gives an isomorphism on the first sheet. It follows that the first differentials also 
coincide. Thus the first differential of our spectral sequence is  as stated
and the higher differentials vanish for dimensional reasons.
\end{proof}

Note, that $F_iC_i(\Ch(\g))$ is spanned by chains $a_0\otimes a_1\otimes\cdots\otimes a_i$
such that $\deg a_{>0}=1$. Taking into account Proposition \ref{tor}
we get the following.

\begin{cor}
	Every cycle in $C_*(\Ch(\g))$ may be presented by a sum of chains 
	$a_0\otimes a_1\otimes\cdots\otimes a_i$
	with $\deg a_{>0}=1$.  
\label{cor}
\end{cor}

Finding  an explicit formula for these cycles  seems to be an interesting question.

\subsection{Invariant vector fields}

Along with the Hochschild complex as above one may consider  
the Hochschild complex $C_*(A,M)$ of a $dg$-algebra  $A$
with coefficients in a $A$-bimodule $M$ (see e.~g. \cite[Ch 1.1]{L}). It is given
by the same formulas (\ref{1}) and (\ref{2}), but 
$a_0$ now is an element of $M$. 
For a ungraded algebra the reduced Hochschild complex calculates $\Tor_*^{A\otimes A^o}(A, M)$.

The $\Ch(\g)$-module of 1-forms $\Ch(\g, \gd)$
is a bimodule as well, because the algebra is supercommutative.
Introduce the Hochschild complex of $\Ch(\g)$
with coefficients in this bimodule $C_*(\Ch(\g), \Ch(\g, \gd))$.

\begin{prop}
    The formulas
	\begin{align}
	\begin{split}
	a_0\otimes a_1\otimes\cdots\otimes a_i &\mapsto a_0\,d_{dR}a_1 \otimes a_2\otimes\cdots\otimes a_i\\
	a_0\otimes a_1\otimes\cdots\otimes a_i &\mapsto \pm a_0\,d_{dR}a_i\otimes a_1\otimes\cdots\otimes a_{i-1},
	\label{at}
	\end{split}
	\end{align}	
	where the sign is defined by the Koszul rule,
	define   morphisms from the Hochschild complex $C_*(\g)$ to the Hochschild complex with coefficients 
	$C_*(\Ch(\g), \Ch(\g, \gd))$ of degree 1.
\end{prop}

\begin{proof}
	This is a direct calculation.
\end{proof}

The following proposition describes these morphisms 
in terms of the quasi-iso\-mor\-phism (\ref{HKR}).

Recall some basic facts from Lie group theory. For a finite-dimensional
Lie algebra $\g$ denote by $U_\g$ its enveloping algebra.
This is a Hopf algebra which is dual to the Hopf algebra of formal functions
$F(G)$ on the formal group associated with $\g$. The Poincar\'e--Birkhoff--Witt
map from the symmetric power of $\g$
to its universal enveloping $i_{PBW}\colon S^*\g \to U_\g$ provides an isomorphism between 
them as adjoint $\g$-modules.
It is dual to the exponential coordinate map $exp^*\colon F(G)\to \Sgd $.
Maps
\begin{equation}
\F_L\colon F(G)\to F(G)\otimes \gd \quad\mbox{and}\quad \F_R\colon F(G)\to F(G)\otimes \gd 
\label{lr}
\end{equation}
dual to the multiplications
$$
U_\g\otimes\g\to U_\g \quad\mbox{and}\quad \g \otimes U_\g\to U_\g
$$
respectively. 
After identifying $G$ and $\g$ by the exponential map,
the maps (\ref{lr}) are given by elements of $Vect(\g)\otimes\gd$.
Corresponding maps from $\g$ to $Vect(\g)$ are given by left and right invariant vector fields on $G$.
Applying the constant trivialization of the tangent bundle to $\g$
one may identify such a tensor with a section of the trivial
vector bundle with fiber $\End(\g)$ over $\g$.   
In other words, this section is the transformation matrix between 
the constant basis of the tangent bundle and the one given by left (right) invariant vector fields.
By e.~g. \cite[Ch. 3.4]{R} they are given by formulas
\begin{equation}
	\id \pm\frac{1}{2}\Ad +\sum_{n\ge 1}\frac{B_{2n}}{(2n)!} \Ad^{2n}
\label{formula}
\end{equation}
("$+$" for the first and "$-$" for the second tensor),
where $\Ad$ is the structure tensor of the $\g$
 considered as linear function on $\g$ taking values in $\End(\g)$
and $B_n$ are Bernoulli numbers:
\begin{equation}
	\sum_{n\ge 0} \frac{B_n}{n!} z^n=\frac{z}{e^z-1}.
\label{bernoulli}
\end{equation}

Recall that Proposition \ref{HKR} identifies $C_*(\Ch(\g))$ with the complex 
$\Ch(\g, \Sgd^{ad})$.
In the same way, one can build a quasi-isomorphism
between $C_*(\Ch(\g), \Ch(\g, \gd))$ and $\Ch(\g, \gd\otimes \Sgd^{ad})$.  

\begin{prop}
Under the quasi-isomorphism as above, maps (\ref{at}) 
$$
\Ch(\g, \Sgd^{ad})\to \Ch(\g, \gd\otimes \Sgd^{ad})
$$
are induced by
(\ref{lr}), where $\Sgd$
is identified with $F(G)$ by the exponential map; 
that is, (\ref{at}) are given by formulas (\ref{formula}).
\label{atiyah}	
\end{prop}

\begin{proof}
Recall that in the proof of Proposition \ref{tor} 
we considered the direct product of terms of the standard complex
calculating $\Tor^{\Ch(\g)}_*(\k,\k)$ and identified it with
$\k[[\gd]]$. Consider also the complex calculating
$\Ext^*_{\Ch(\g)}(\k,\k)$, where we  take direct sum rather that direct product. The former complex is dual to the latter one. As in the proof of Proposition \ref{tor},
the spectral sequence argument shows, that the
cohomology of the latter complex is isomorphic to $S^*(\g)$. The Yoneda product endows it with multiplication which, as it easy to check,
gives it the structure of the universal enveloping algebra of $\g$. As the unbounded 
version of $\Tor^{\Ch(\g)}_*(\k,\k)$ is dual to it, this is
formal functions on the group. The quasi-isomorphism
(\ref{hkr}) is dual to the PBW isomorphism; that is, it is given by the exponential coordinates.  Formulas (\ref{at}) define  the left and right actions
of the Lie algebra on the functions on the group. This proves the statement.
\end{proof}

\begin{remark}
Maps (\ref{at}) may be thought as the Atiyah class of the diagonal
of the $dg$-manifold  which is a spectrum of $\Ch(\g)$.
Analogous maps and formulas  for a usual complex manifold
play a crucial role in \cite{M1}.
\label{a}
\end{remark}

\subsection{Quantization}

Let now $\g$ be an finite-dimensional Lie algebra with a non-degenerate invariant scalar product
$\langle\cdot\,, \cdot\rangle$. The scalar product may be thought of as a constant symplectic structure  of degree $-2$ on the $dg$-manifold  (or Q-manifold), which is the spectrum of $\Ch(\g)$. That is, we define a Poisson bracket on
$\Ch(\g)$ on the generators by $\{x,y\}=\langle x,y\rangle$
and extend it to the whole algebra by the Leibnitz rule.
In terms of \cite{AKSZ} we get 
a $QP$-manifold.

A symplectic structure gives a first order
deformation of the product of functions on a manifold and thus
deforms the Hochschild complex. Our aim is to calculate it in our case.
 
More precisely, consider the ring $\k[\varepsilon]$, where $\deg\varepsilon=2$ and 
$\varepsilon^2=0$
and the Chevalley--Eilenberg complex $\Ch(\g)\otimes\k[\varepsilon]$
over $\k[\varepsilon]$
with the differential as before, with the product given  
by $x\cdot y=x\wedge y +\frac{1}{2}\varepsilon\langle x, y\rangle$.
Take the Hochschild complex of $\k[\varepsilon]$-algebra $\Ch(\g)\otimes\k[\varepsilon]$, that is, all tensor products
are taken over $\k[\varepsilon]$. It is a module over $\k[\varepsilon]$.
Multiplication by $\varepsilon$ defines a 2-step filtration on it.
Consider the spectral sequence associated with this filtration.
The 0-th sheet is $C_*(\Ch(\g))\otimes\k[\varepsilon]$.
The following proposition describes $d_0$ of this spectral sequence,
which is the first order deformation of the differential in the Hochschild complex. 

\begin{prop}
Contract tensors (\ref{lr}) from  $Vect(\g)\otimes\gd$ with
the pairing $\langle\cdot\,, \cdot\rangle$ and consider
the resulting element of $Vect(\g)\otimes\g$
as a differential operator on $\Ch(\g,  \Sgd^{ad})$
of the second order, where term $\cdot\otimes\g$
differentiates $\Ch(\g)$ and term $Vect(\g)\otimes \cdot$
differentiates $\Sgd$.
Under quasi-isomorphism (\ref{hkr}) differential $d_0$ of the
above-mentioned spectral sequence is given by 
half-sum of these operators on the complex $\Ch(\g,  \Sgd^{ad})$.
By (\ref{formula}), the matrix of this differential operator is given by
\begin{equation}
	\id  +\sum_{n\ge 1}\frac{B_{2n}}{(2n)!} \Ad^{2n},
\label{formula2}
\end{equation}	
$B_n$ are Bernoulli numbers,  $\Ad$ is the structure tensor of the $\g$, 
being considered as linear function on $\g$ taking values in $\End(\g)$.
\label{At}
\end{prop}  

\begin{proof}
By the very definition, the derivative of the differential of the Hochschild
complex along the first order deformation given by a symplectic form is presented by the formula
\begin{align}
&d_0(a_0\otimes a_1\otimes\cdots\otimes a_n)=\frac12\{a_0,a_1\}\otimes a_2\otimes\cdots\otimes a_n-\nonumber\\
&-\frac12a_0\otimes \{a_1,a_2\}\otimes\cdots\otimes a_n+\dots \pm \frac12\{a_n, a_0\}\otimes a_1\otimes\cdots\otimes a_{n-1},
\label{poiss}
\end{align}
where $\{\,,\,\}$ is the Poisson bracket, associated with the symplectic form.
Apply it to the Chevalley--Eilenberg complex.
By Corollary \ref{cor}, any class in $C_*(\Ch(\g))$ may be represented
by a cycle with degree one elements as entries with non-zero indexes.
As the Hochschild complex is reduced, it follows that in (\ref{poiss}) only the first and the last term do not vanish. These terms are given by the maps
(\ref{at}). Applying Proposition \ref{atiyah} we complete the proof. 
\label{d0}
\end{proof}

Proposition \ref{At} defines, therefore, on the algebra $\Ch(\g, \Sgd^{ad})$
a differential operator of order 2 and of cohomological degree $-1$. 
On this algebra another  differential
operator of the same order and degree is defined, 
in terms of the above proposition  it is given by the unit matrix.
Call it the Brylinski differential after \cite{Br} and denote it by $d_{Br}$
They are not chain homotopic, but by the following
proposition they  become  such after conjugation by an automorphism 
of complex $\Ch(\g, \Sgd^{ad})$. This automorphism equals to  multiplication by
the Duflo character.   

Given a Lie group $G$, equip it with the left invariant volume form (which is the right invariant as well, due to the invariant scalar product). Equip 
its Lie algebra $\g$ with the constant volume form and denote by
$\j\in\Sgd$ the Jacobian of the exponential map.
The Duflo character is the power series on $\g$ 
which is the square root of the Jacobian
and is given by
\begin{equation}
\j^{\frac12}=\exp \,\sum_{n=1}^\infty \frac{B_{2n}}{4n (2n)!} \,\Tr( \Ad^{2n}),
\label{duflo}
\end{equation}  
where $B_n$ are the Bernoulli numbers from (\ref{bernoulli})
and $\Ad$ is the linear function on $\g$ taking values in $\End(\g)$ as above. 

\begin{prop}
Under the quasi-isomorphism (\ref{hkr}), the differential $d_0$ on $\Ch(\g, \Sgd^{ad})$
is chain homotopic to 
$\j^{-\frac12}\circ d_{Br}\circ\j^{\frac12}$,
where $\j^{\frac12}$ is the operator of the multiplication of $\Sgd$
by the Duflo character and $\j^{-\frac12}$ is the inverse operator.
\label{char}
\end{prop}

\begin{proof}
We will use the differential operator notation for endomorphisms of complex $\Ch(\g, \Sgd^{ad})$
and the Einstein summation convention.
For example, $d_{Br}= g_{ij}\partial/\partial{x^i}\,\,\partial/\partial\,{d_{dR}x^j}$,
where $g_{ij}$ is the scalar product, $x_i$ is a basis in $\gd$
 and $d_{dR}$ is the de Rham differential (we think of $\Ch(\g, \Sgd^{ad})$
as of differential forms on $\Ch(\g)$ as in the first section).
By Proposition \ref{d0}, 
\begin{equation}
d_0-d_{Br}=\sum_{n\ge 1}\frac{B_{2n}}{(2n)!} (\Ad^{2n})_j^i 
g_{ik}\partial/\partial{x^k}\,\,\partial/\partial\,{d_{dR}x^j},
\label{dbr}
\end{equation}
where $g_{ij}$ is the scalar product and ${\Ad^*}$ is the element of $\Sgd\otimes \End(\g)$.
Consider the differential operators of order $2$ given by
$$
H_{2n-1}=(\Ad^{2n-1})_j^i g_{ik}\partial/\partial{x^k}\wedge\partial/\partial{x^j}.
$$
We leave to the reader to check that
$$
[d_{CE}, H_{2n-1}]=  2(\Ad^{2n})_j^i g_{ik}\partial/\partial{x^k}\,\,\partial/\partial\,{d_{dR}x^j}-
\frac{1}{2n}[d_{Br}, \Tr( \Ad^{2n})],
$$
where $d_{CE}$ is the differential in the Chevalley--Eilenberg complex; all other terms vanish 
due to the Jacobi 
identity. Comparing it with (\ref{dbr}) we see, that $d_0-d_{Br}$ is chain homotopic to $[d_{Br}, 
\ln 
j^{\frac12}]$.
This implies the  statement.
\end{proof}

\begin{remark}
The above proposition can be stated and proved  in a coordinate-free manner
for any $QP$-manifold in terms of \cite{AKSZ}.
In the setting  of \cite{M1} (see Remark \ref{a}) 
it describes the differential on the differential forms on a complex symplectic
manifold, that is, on the Hochschild homology of the structure sheaf, coming from the 
first order deformation
of the structure sheaf along the symplectic structure. It  seems that
when applied to the cotangent bundle of a complex manifold,
it gives an alternative way of calculating the Todd class of this manifold.
\label{remark}
\end{remark}

\begin{remark}
Proposition \ref{char} was inspired by the proof of the Duflo isomorphism
for a Lie algebra with an invariant  scalar product from \cite{AM}. As we will see below,
the calculation above is connected with another proof of the Duflo isomorphism, the one 
from \cite{BN}. 
\end{remark}

\section{$e_n$-algebras}

\subsection{$e_n$-algebras}

The main character in what follows is a unital algebra over the operad $e_n$,
the operad of rational chains of the little discs operad.
Recall that this $dg$-operad and its cohomology for $n>1$
is the shifted Poisson operad, which is generated by an associative
commutative product $\cdot$ of degree 0 and a Lie bracket $\{\,,\,\}$ 
of degree $1-n$, they subject to the Leibnitz rule.
A $e_\infty$-algebra is a unital homotopy commutative algebra and $e_0$-algebra is a complex with 
a chosen cocycle.

The embedding of spaces of little discs induces the map of operads $e_k\to e_n$ for $k<n$.
It induces a functor from $e_n$-algebras to $e_k$-algebras
which we denote by $\obl^n_k$. In particular, functor $\obl_n^\infty$
produces an $e_n$-algebra from any commutative ( that is, $e_\infty$-) algebra.

For our purpose it will be more convenient to consider the operad of
rational chains of the Fulton--MacPherson operad, see \cite{W} and references therein for details.
The latter operad is homotopy equivalent to $e_n$ and below we will make no difference  
between them; by saying an $e_n$-algebra we shall mostly mean an algebra
over the Fulton--MacPherson operad.

The operations of the operad of little discs are spaces of $n$-balls embedded
in a radius one $n$-ball. The group $SO(n)$ acts by rotations on the big ball.
In order to take this action into the account one may consider $SO(n)$ as an
operad with 1-ary only operations and take the semi-direct product
of this operad and the little discs operad. The result is called the framed
little discs operad, see \cite{SW}. We denote the $dg$-operad of chains of this
operad  by $fe_n$. 

An alternative and better way to take into account the $SO(n)$-action
is to consider equivariant chains. It gives us a $dg$-operad colored by $BSO(n)$,
see e.~g. \cite{W}. Modules over this operad are $SO(n)$-equivariant complexes.
Call these modules  equivariant $e_n$-algebras. 
In general, the category of such algebras is not the same as the one
of $fe_n$-algebras. However, for $n=2$ commutativity of the group  simplifies
things and these categories are essentially the same.

Consider the latter case in some detail. The cohomology of $fe_2$ is known as 
the  Batalin--Vilkovisky  (BV) operad, see e.~g. \cite{SW}. It is generated
by the product $\cdot$ and the bracket $\{\,,\,\}$ obeying the same relations as
those in $e_2$ and an additional $1$-ary operation $\Delta$ of degree $-1$
obeying the relations
$$
\Delta^2 =0, \qquad \{a,b\} = (-1)^{\left|a\right|}\Delta(ab) - (-1)^{\left|a\right|}\Delta(a)b - a\Delta(b).
$$

\subsection{The factorization complex}

Given a framed $n$-manifold (that is, a manifold  with the tangent bundle trivialized) $M$
and a $e_n$-algebra, the factorization complex $\int_M A$ is defined
as in  \cite{W} and in the references therein. The idea of the definition is straightforward:
discs embedded  in $M$ define a right module over $e_n$ and the factorization complex
is the tensor product over $e_n$ of this right module with the left module given by $A$. 

In order to extend the above  definition to unframed manifolds, one needs the algebra $A$
to be equivariant. Locally, one may choose a framing on $M$ and apply the definition
and then use the equivariance to identify results for different framings.  

One important property of the factorization complex is its behavior with respect
to gluing, see e.~g. \cite{G} and references therein. Let $M_1$ and $M_2$
be two manifolds with isomorphic boundaries $B$. Then for a $e_n$-algebra $A$ there is 
a map of complexes
$$
\int_{M_1} A\otimes\int_{M_2} A\to \int_{M_1\cup_B M_2} A.
$$
It follows that for $k<n$, a $k$-manifold $M^k$ and a $e_n$-algebra $A$, 
the complex $\int_{M^k\times I^{n-k}}A$ is a $e_k$-algebra,
and it is equivariant, if $A$ is. In particular, for an $n$-manifold $M$
with boundary $B$ the complex   $\int_{B\times I}A$ is a (homotopy) algebra, and
the map above equips $\int_M A$ with a module structure over it.
In terms of this action, the gluing rule may be written as
\begin{equation}
\int_{M_1\cup_B M_2} A=\int_{M_1} A\,\tens{\int_{B\times I}A} \,\int_{M_2} A.
\label{glue}
\end{equation}

Another important property of the factorization complex is a kind of homotopy invariance:
$$
\int_{M^k\times I^{n-k}}A=\int_{M^k} \obl\nolimits^n_k A.
$$
Below we will make no difference between the two sides of this equality
and will denote them simply by $\int_{M^k}  A$.
In particular, the factorization complex on a disk is quasi-isomorphic, as a complex,
to the algebra itself.

\begin{example}
Let $A$ be an equivariant $e_2$-algebra. Then its factorization complex
on the disc $\int_{D^2}A$, which is $A$ itself, is a module
over $\int_{S^1\times I^1} A=\int_{S^1}\obl_1^2 A$, which is the Hochschild homology complex 
of $\obl_1^2 A$.
The equivariance of $A$ is essential here: without it, the
Hochschild complex of $e_2$-algebra $A$ does not act on $A$, and,
if an equivariance structure is chosen, the action depends on this choice.
In order to see it, note that $S^1\times I^1$ is a framed manifold, that is
why we do not need equivariance to 
take its factorization complex for any, not only equivariant algebra.
However, this framing, which comes from the constant framing on the square
after gluing together two opposite edges, can not be extended to the whole
disc obtained  
 from the annulus $S^1\times I^1$  by gluing one of its boundary circles with the disc. 
Hence, in order to construct the desired
action by gluing the annulus with the disc one  need to identify
factorization complexes with different framings, and here one needs the equivariance.
\label{annulus}
\end{example}

\subsection{Weyl $n$-algebras}

The type of equivariant $e_n$-algebras we need are the Weyl $n$-algebras; we refer
to \cite{W} and \cite{CPTVV} for the definition. In order to build such an algebra one needs a 
super-vector
space $V$ with a super-skew-symmetric non-degenerate bilinear form on it.
The $e_n$-algebra associated with such data is denoted by $\We(V)$.
In analogy with the usual Weyl algebra, it is the deformation of the polynomial
algebra generated by $V$ in the direction given by pairing.
In fact, this is an algebra over the field of Laurent formal series
in the quantization parameter $h$; this, however, must be ignored, 
assuming, loosely speaking, that $h=1$.

There are some important properties we need. Firstly,
considered as an $e_k$-algebra, where $k<n$, it is commutative.
In other words, $\obl_k^n \We(V)=\obl_k^\infty \k[V]$ for any $k<n$,
where $\k[V]$ is the polynomial algebra.

The following property is crucial for the construction of the perturbative invariants in \cite{W}: 
for any $n$-manifold $M$ the complex $\int_M\We(V)$ has one dimensional cohomology (\cite[Proposition 11]{W}). 
I conjecture that, for any $k<n$, the factorization complex
$\int_{N^k\times I^{n-k}} \We(V)$ 
is again a Weyl algebra
for any $k$-dimensional manifold $N^k$.

\begin{example}
Let $V$ be a vector space.
Equip $V\oplus V^\vee[-1]$ with the standard form of degree $-1$.
Then $\mathop{\mathcal{W}^2}(V\oplus V^\vee[-1])$ is the space of polyvector fields on
$V^\vee$ and standard operations on it --- the Gerstenhaber bracket and the cup product --- are the operations
of the cohomology of $e_2$.

As any Weyl algebra,  $\mathop{\mathcal{W}^2}(V\oplus V^\vee[-1])$ is equivariant.
Thus it is acted on by the operad $fe_2$ and by its cohomology, which is the BV operad.
The operation $\Delta$ is equal to the de Rham differential, where the polyvector fields
are identified with the differential forms by means of the constant volume form.
Another choice of the volume form leads to another $fe_2$-structure with the same
underlying $e_2$-structure. 
\end{example}

\subsection{The action}

For  associative (or $e_1$-) algebras the notion of modules plays the central role.
The higher generalization of this notion is a $e_n$-algebra acting on a $e_{n-1}$-algebra,
for the definition and the discussion see e.~g. \cite{G} and references therein.  
Constructively, it may be defined by means of the Swiss cheese operad, which is 
especially convenient for algebras over the operad of chains of the Fulton--MacPherson 
operad. In the same way as the operations  of the little discs operad are given 
by the configuration spaces of $\mathbb{R}^n$, the operations of the Swiss cheese
operad are given by the spaces of distinct points in $\mathbb{R}^{\ge 0}\times \mathbb{R}^{n-1}$.
There are points of two types: those on the boundary and those in the interior.
This gives a colored operad with two colors.
If an $e_n$-algebra $B$ acts on an $e_{n-1}$-algebra $A$, then elements
of $B$ sit on the interior points and elements of $A$ --- on boundary points. 
For further details we refer the reader to \cite{V}. 

Note that the action of the Swiss cheese operad may be formulated in terms
of factorization sheaves; for the definition of the latter see e.~g.
\cite{G} and references therein. Namely, such an action 
is equivalent to a factorization sheaf on the half-space such that
its restriction to the boundary and to the interior are constant factorization sheaves,
corresponding to the $e_{n-1}$-algebra $A$ and the $e_n$-algebra $B$. 

It is known that for any $e_n$-algebra there exists a universal $e_{n+1}$ algebra
$\End(A)$ acting on it (\cite{Lu}). In other words,  an action of an $e_{n+1}$
algebra $B$ on $A$ is the same as  a morphism of $e_{n+1}$-algebras
$B\to \End(A)$. For an associative (or $e_1$-) algebra the $\End$-object
is its Hochschild cohomology complex.

Let $V$ be a vector space.
Equip $V\oplus V^\vee[1-n]$ with the standard form of degree $(1-n)$.
Then $\We(V\oplus V^\vee[1-n])$ is $\End(\k[V])$, where $\k[V]$ is the polynomial
algebra. In order to see it, one may construct an action of $\We(V\oplus V^\vee[1-n])$ 
on $\k[V]$ directly by using the 
Swiss cheese operad
and the Fulton--MacPherson compactification. Then one need to check
that the resulting map $\We(V\oplus V^\vee[1-n])\to \End(\k[V])$ is a quasi-isomorphism.

This action commutes with taking the factorization complex.
That is, if an equivariant $e_{n+1}$-algebra $B$ acts on an equivariant $e_n$-algebra $A$, then
for a $k$-manifold  $N$ the $e_{n-k+1}$-algebra $\int_{N^{k}\times I^{n-k+1}} B$
acts on $e_{n-k}$-algebra $\int_{N^{k}\times I^{n-k}} A$. It follows immediately from  
definitions of the Swiss cheese operad and of the factorization complex.
It seems plausible that  under appropriate conditions  $\int_{N^{k}\times I^{n-k+1}} \End(A)=\End(\int_{N^{k}\times 
I^{n-k}} A)$.

\begin{example}
Consider the polynomial algebra $A=\k[V]$ as an associative algebra. Its Hochschild cohomology complex
$C^*(A,A)$ 
 (which, as it was 
mentioned above, is $\mathop{\mathcal{W}^2}(V\oplus V^\vee[-1])$)
acts on it. It follows, that $\int_{S^1} C^*(A,A)$, which is a $e_1$-algebra,
acts on $\int_{S^1} A$. The latter complex is the Hochschild homology complex 
of $A$, which is known to be quasi-isomorphic to 
the direct sum of shifted differential forms (see e.~g. \cite{L}). 
It is shown in \cite{NT} that the first complex is  quasi-isomorphic to the differential operators on differential forms, and this is in good agreement
with the speculation preceding the present example.
\label{diff}
\end{example}

Recall, that in  Example \ref{annulus} for any an equivariant $e_2$-algebra $A$ we construct
action of $e_1$-algebra   $\int_{S^1} A$ on the underlying complex of $A$.
In the same way for any  equivariant $e_n$-algebra $A$ the
$e_1$-algebra   $\int_{S^{n-1}} A$ acts on the underlying complex of $A$:
the action is given by gluing a $n$-ball and $S^{n-1}\times I$.
It may be generalized even further. The factorization complex   $\int_{S^k} A$,
which is a $e_{n-k}$-algebra,
analogously acts on $e_{n-k-1}$-algebra $\obl^n_{n-k-1} A$.
As this action plays a crucial role in the next Section,
let us phrase it below as
the construction.

\begin{construction}
\label{construction}
Let $A$ be an equivariant $e_n$-algebra. Then, for any $k<n$, the
$e_{n-k}$-algebra $\int_{S^k} A$
naturally acts on $\obl^n_{n-k-1} A$. 
The corresponding action of the Swiss cheese operad
is defined as follows.
Embed $\mathbb{R}^{\ge 0}\times \mathbb{R}^{n-k-1}$ linearly into $\mathbb{R}^n$. Put at any point of this half-space the factorization complex 
of $A$ on the $k$-sphere lying into the $k+1$ space
perpendicular to the half-space, with its center on 
$0\times \mathbb{R}^{n-k-1}$ and passing through this point.
In particular, for points on $0\times \mathbb{R}^{n-k-1}$
we get the sphere of zero diameter, that is a point and the factorization 
complex is $A$ itself. 

In other words, consider a map $\mathbb{R}^n\to\mathbb{R}^{\ge 0}\times \mathbb{R}^{n-k-1}$ which sends a point to the pair which  
consists of the distance from the point to the subspace $\{0\}\times \mathbb{R}^{n-k-1}$
and the orthogonal
projection on $\mathbb{R}^{n-k-1}$. Then the direct image of the factorization sheaf on $\mathbb{R}^n$  corresponding to $A$
is the desired factorization sheaf 
on  $\mathbb{R}^{\ge 0}\times \mathbb{R}^{n-k-1}$.
\end{construction}

\section{Wilson loop}

\subsection{Quantum Chevalley--Eilenberg algebra}
\label{Chq}

Given a Lie algebra $\g$ with an invariant scalar product, 
in \cite[Appendix]{W} (see also \cite[3.6.2]{CPTVV}) a $e_3$-$dg$-algebra  $\Chq(\g)$ is defined as follows.  
Take the Weyl $3$-algebra given by the space $\gd[1]$ with the scalar product
and equip it with a differential  $\frac{1}{h}\{\cdot,q\}$,
where $\{\,,\,\}$ is the image of the Lie bracket under the map $L_\infty\to e_3$
(see e.~g. \cite[Proposition 2]{W}) and $q$ is the degree $3$ element, which is the composition of the Lie  bracket
on $g$ and the scalar product. Call this $e_3$-algebra
the quantum Chevalley--Eilenberg algebra.

Consider the Hochschild  complex $C_*(\Chq(\g))$. Here and in what follows
we will consider unbounded Hochschild chains, that is, the Hochschild
complex which is the direct product of its terms. 

This Hochschild complex is the factorization complex $\int_{S^1} \Chq(\g)$.
As $\Chq(\g)$ is $e_3$-algebra, the Hochschild complex is an $e_2$-algebra.
Consider it as an $e_1$-algebra, that is take $\obl_1^2\int_{S^1} \Chq(\g)$.
By the very definition it is equal to $\int_{S^1} \obl_2^3\Chq(\g) $.
We mentioned above an 
important property of Weyl algebras: 
 $\obl_k^n \We(V)=\obl_k^\infty \k[V]$ for any $k<n$.
It follows, that $\obl_2^3\Chq(\g)=\obl_2^\infty\Ch(\g)$ .
Thus $\obl_2^3\Chq(\g)$ is just the super-commutative
Chevalley--Eilenberg algebra. Its Hochschild complex
is again a super-commutative algebra quasi-isomorphic to $\Ch(\g,\Sgd^{ad})$ by Proposition \ref{HKR}.To recap,
$\int_{S^1} \Chq(\g)$ as $e_1$-algebra, that is $\obl_1^2\int_{S^1} \Chq(\g)$
is isomorphic to  $\Ch(\g,\Sgd^{ad})$.

Now, let us apply the construction from the previous section
to $A=\Chq(\g)$, $n=3$ and $k=1$. It gives an action of
the $e_2$-algebra $\int_{S^1} \Chq(\g)$ on $\obl_1^3\Chq(\g)$,
which is $\obl^\infty_1\Ch(\g)$.
That is we get a map from the $e_2$-algebra
$\Ch(\g,\Sgd^{ad})$ to the Hochschild
cohomology complex of $\Ch(\g)$ by the universal property,
which is easily seen to be a quasi-isomorphism.
The Hochschild
cohomology complex of $\Ch(\g)$ is known to be equal to $\Ch(\g, U_\g^{ad})$,
where $U_\g$ is the universal enveloping algebra of $\g$.

To be more precise, in this way we get a map from $\Ch(\g,\Sgd^{ad})$
to  $\Ch(\g, U_\g^{ad})\otimes \k[[h]]$. 
The $e_1$-structure on this complex comes from the 
one on the universal enveloping algebra. On the other
hand, as it is shown in the previous paragraph,
$\int_{S^1} \Chq(\g)$ as $e_1$-algebra isomorphic
to  $\Ch(\g,\Sgd^{ad})$. Thus, an explicit form
of this map, which is 
supplied by the proposition below, implies the Duflo isomorphism.
 
\begin{prop}
The map of complexes 
\begin{equation}
\Ch(\g,\Sgd^{ad})=\int_{S^1} \Chq(\g)\to \Ch(\g, U_\g^{ad})\otimes \k[[h]]
\label{map}
\end{equation}
as above 
is chain homotopic to the map induced by the composition
\begin{equation}
\Sgd \stackrel{\exp {(h(\cdot,\cdot))}}{\longrightarrow} S^*\g\otimes \k[[h]]
\stackrel{\j^{\frac{1}{2}}}{\longrightarrow} S^*\g\otimes \k[[h]]
\stackrel{PBW}{\longrightarrow} U_\g \otimes \k[[h]],
\label{chain}
\end{equation}
where the first arrow is given by the scalar product multiplied by $h$, the second is the
contraction with the Duflo character (\ref{duflo})
and the third one is the PBW map.
\label{superduflo}
\end{prop}

\begin{proof}[Sketch of proof]
As it was mentioned above, $\Chq(\g)$ as an $e_2$-algebra is isomorphic to the commutative
algebra $\Ch(\g)$. 
It follows that the  map induced by the unit embedding
$\Ch(\g)\to \Ch(\g,\Sgd^{ad})$ is a morphism of $e_2$-algebras and in composition
with (\ref{map}) it gives the standard map
$\Ch(\g)\to \Ch(\g, U_\g^{ad})$ . Thus we know the image of the subalgebra $\Ch(\g)$ under (\ref{map}).
One may see that the whole map (\ref{map}) may be uniquely determined
from it as the unique extension compatible with the Lie bracket coming from the $e_2$-structure.
To see this one may use the faithful action of $\int_{S^1\times S^1}\Chq(\g)$  on  $\int_{S^1} \Chq(\g)$ as in the sketch of the proof of Proposition \ref{last}.

So our immediate purpose is to calculate the bracket on $ \Ch(\g, \Sgd)$,
which is $\int_{S^1} \Chq(\g)$.
As we will see below, it is enough  to calculate the bracket with an
element which is image of $a\in \Ch(\g) $ under the embedding map as above.
Given an element $b \in\int_{S^1} \Chq(\g)$, the bracket
$\{a,b\}$ may be interpreted geometrically as follows. 
Consider the solid torus $D^2\times S^1$ and  two circles in it:
$C=(0, S^1)$, call it the big one, and $c=(\{x\in D^2\mathrel{}\mid\mathrel{}|x|=1/2\}, *)$, call it the small one. The cycle in the factorization complex of the solid torus,
which is $\int_{S^1} \Chq(\g)$,
representing  $\{a,b\}$ equals  $C_b\otimes ([c]\otimes a)$,
where by $C_b$ we denote the image of $b$ in $\int_{D^2\times S^1} \Chq(\g)$
under  the embedding $C\hookrightarrow D^2\times S^1$.
One may see that cycle $[c]\otimes a$ is equal to $c_{d_{dR}a}$,
where $d_{dR}$ is the de Rham differential. If $a=x_1\wedge\dots\wedge x_i$,
then $d_{dR}x=\sum \pm d_{dR} x_i\, x_1\wedge\dots\hat{x_i}\dots\wedge x_n$.

Let us now start pulling the small  circle to unlink it from the big one.
That is, consider a family of cycles $c^t$
where $c^t$ is a family of circles in the solid torus such that
$c^0$ is the small circle, $c^1$ is a circle unlinked with the big circle
and only one circle in the family intersects the big one.  
Until the circles do not intersect, nothing happens and the cycle
$C_b\otimes c_a^t$ remains in the same class.
But,  as soon as they intersect each other, this class is changed by the class which is
a derivation of $b$. The calculation shows that for 
$b=d_{dR}x_0\, x_1\wedge\dots\wedge x_n$ it is given 
by the sum of maps (\ref{at}) contracted with $x_0$ and multiplied by 
$x_1\wedge\dots\wedge x_n$.
The reasoning is analogous to Proposition \ref{At}: unlinking influences only
around the intersection point.
When the small circle is unlinked from the big one,
$C_b\otimes c_a^1$  vanishes, because $c_a^1=[c^1]\otimes a$ is a boundary.

Note, that the $e_2$-algebra $\Ch(\g,\Sgd^{ad})$  is, in fact, a $fe_2$-algebra.
Thus, instead of the Lie bracket, one may calculate the operator $\Delta$
corresponding to the rotation. 
Given an element $x=\sum a_ib_i\in \Ch(\g,\Sgd^{ad})$, where $a_i$ are
in the odd part and $b_i$ in the even part, one may show, that
$$
\Delta x= \sum\{a_i, b_i\}.
$$
Apply the calculations from the previous paragraph to it.
Comparing it with Proposition \ref{At} we see, that the operator $\Delta$
on $\Ch(\g,\Sgd^{ad})$ coincides with the operator $d_0$ from there.
Proposition $\ref{char}$ implies that  the Duflo character
gives an isomorphism between this operator and $d_{Br}$.
In order to complete the proof, one has to verify
that $d_{Br}$ is the operator $\Delta$ for the $fe_2$-algebra $\Ch(\g, U_\g^{ad})$.
\end{proof}

While proving the proposition we found that the operator $\Delta$
on the $fe_2$-algebra $\int_{S^1} \Chq(\g)$ is equal to the first order deformation of the Hochschild differential
of $\Ch(\g)$ that we discussed in the first section.
I have no explanation for  this coincidence.

\subsection{Invariants of knots}

In \cite{W} we constructed invariants of manifolds using Weyl $n$-algebras.
Below we develop this idea for manifolds with embedded links.
Let us restrict ourselves to a $3$-sphere with a knot in it.

As it was observed in \cite{W}, the cohomology of the factorization complex
of the Weyl $n$-algebra $\We(V)$ on a closed $n$-manifold is one-dimensional.
If   $V$ lies in degree $1$ and the manifold is a $3$-sphere (or a homology sphere),
then the generator of this cohomology is given by the class $[p]\otimes S^{\mathrm{top}}V$,
where $p$ is a point in the manifold. As it was explained in \cite[Appendix]{W},
the factorization complex $\int_{S^3} \Chq(\g)$ is isomorphic to the complex of
the underlying Weyl $3$-algebra. Since the Chevalley--Eilenberg
differential is inner, one needs to consider here unbounded chains
that is, take direct product rather than the direct sum.
It is easy to see that the generator in the cohomology of $\int_{S^3} \Chq(\g)$
is given by $[p]\otimes S^{\mathrm{top}}\gd$. Call it the standard cycle.
The idea of invariants we construct is to produce another cycle and compare it
with the standard one.

Given a knot $K\colon S^1\hookrightarrow S^3$ and a class 
$f\in \int_{S^1}\Ch(\g)=\Ch(\g, \Sgd^{ad})$,  denote by 
$K_f$ the direct image of this class under $K$.
The class we are interested in is $([p]\otimes S^{\mathrm{top}}\gd)\otimes K_f$.
For dimensional reasons, only $f$ of degree 0 are interesting,
 in fact, $f\in \Sgd^{inv}$. Thus we get the following definition.  

\begin{definition}
For a knot $K$ in $\mathbb{R}^3$ the 
Wilson loop invariant	is the function on $\Sgd^{inv}$
given by
$$
f\mapsto ([\infty]\otimes S^{\mathrm{top}}(\gd[1]))\otimes K_f\in \int_{S^3}\Chq(\g),
$$
where we identify $\int_{S^3}\Chq(\g)$ with $\k[[h]]$ using the standard cycle as the generator.
\end{definition}

In \cite{W} it is  showed that invariants constructed  there are described by formulas
similar to formulas for the Axelrod--Singer invariants. Following the same line,
we see that the Wilson loop invariants are connected with Bott--Taubes invariants;
for a survey of the latter see e.~g. \cite{Vo}. 
There is another invariant of knots ---
the Kontsevich integral, see \cite[Part 3]{CD}.
In principle, it should coincide with the Bott--Taubes invariants, see \cite{Ko}. 
As far as I know, this point is not clear,
for discussion see \cite{Le}.  One may hope that the definition above 
will help to elucidate this.

Our construction of the Wilson loop invariant depends on the choice of a Lie algebra
with a scalar product. One
may give a more complicated, but universal definition of these invariants 
with values in the graph complex, which is the Chevalley--Eilenberg
complex of Hamiltonian vector fields, in the same way as it is outlined in \cite[Appendix]{W}. 

An interesting property of the Kontsevich integral is its value on the unknot: 
it is equal to the Duflo character and
this  allows to prove the Duflo isomorphism,
see \cite{BN} and \cite[Ch. 11]{CD}. 
The following proposition states that the Wilson loop invariant shares this property.

\begin{prop}
The Wilson loop invariant of the unknot is equal to the composition
$$
\Sgd^{inv}\hookrightarrow\Sgd \to U_\g \otimes \k[[h]] \to \k[[h]],
$$
where the second arrow is given by (\ref{chain}) and the third one is the standard augmentation.
\label{last}
\end{prop}

\begin{proof}[Sketch of proof]
As it was discussed in Subsection \ref{Chq}, the $e_2$-algebra $\int_{S^1}\Chq(\g)$
acts on the $e_1$-algebra $\Chq(\g)$. In Proposition \ref{superduflo}
it is shown that this action is not ``naive'', the morphism to
the $\End$-object is the composition of the pairing and the Duflo character.
As it was mentioned above, this action is compatible with taking the factorization complex:
as $\int_{S^1}\Chq(\g)$ acts $\Chq(\g)$ so
$\int_{S^1\times S^1}\Chq(\g)$
acts on  $\int_{S^1} \Chq(\g)$. 
The $e_1$-algebra $\int_{S^1\times S^1}\Chq(\g)$
is the algebra of differential operators on $\Chq(\g, \Sgd^{ad})$,
see also Example \ref{diff}.
The complex $\int_{S^1} \Chq(\g)$ is a kind of a holonomic module
over these differential operators.
But, again, it is not ``naive'', this action is twisted by the Duflo character.

Cut $S^3$ in two solid tori in the standard way, being the infinity point inside one of them
and the unknot is the middle circle of the other.
Now apply (\ref{glue}) to calculate the Wilson loop invariant of the unknot.
As it was mentioned, $\int_{S^1\times S^1}\Chq(\g)$
is the algebra of differential operators and the factorization complexes 
of solid tori are kind of holonomic modules with transversal characteristic varieties.
Now, the calculation of the Wilson invariant is reduced to taking
the derived tensor product of these modules and comparing
cycles in the result given by  different $f\in \Sgd^{inv}$.
Taking into account the Duflo twisting we get the result.
\end{proof}

There is another natural approach to the knot invariants mentioned in
\cite{AFT}.  Given a knot in a closed manifold, one may
cut out a small solid torus around it to get a manifold with boundary.
Then factorization complex for a $fe_3$-algebra of this manifold is a module over
the factorization complex of the boundary torus, which is an invariant of the knot.

The proof of the above proposition makes clear what  happens
when the $fe_3$-algebra is $\Chq(\g)$.
In this case, the pair consisting of an algebra and a  module itself does not depend on the knot;
they are, essentially, the algebra of differential operators and the standard holonomic module over it.
But this module contains a chosen element (it is a $e_0$-algebra!),
which is the image of the unit. And the module together with this element is the invariant of
the knot. Reasoning analogous to the proof of Proposition \ref{last}
shows that this invariant is, essentially, equivalent to the Wilson loop invariant.

\subsection{Skein algebra}

It \cite{Tu} for a Riemann surface $S$ a skein algebra was introduced.
It is generated by non self-intersecting loops on $S$.
We claim that there is a map from this algebra to
$\int_{S\times I}\Chq(\g)$. An element corresponding to a loop $L$
maps to $L_\eta$, where $\eta\in\int_{S^1} \Chq(\g)$ is the canonical element,
which is the preimage of $\exp(hc)$ under (\ref{chain}),
where $c$ is the Casimir element given by the scalar product.

The reason to propose it is the following.
The skein algebra  is the quantization of a Poisson algebra.
The latter appears in \cite{Go,Wo} as a subalgebra of the Poisson algebra
of functions on the moduli space of $G$-local systems on $S$.
But $\int_{S\times I}\Chq(\g)$ must be thought of as the quantization
of the latter Poisson algebra, see \cite{CPTVV}.

Elements $L_\eta$ play an important role since they are generating functions of Dehn twists.
In other words, the cobordism corresponding to the Dehn twist
gives a bimodule over $\int_{S\times I}\Chq(\g)$,  according to 
speculations in the end of the previous subsection.
Then the element $L_\eta$, corresponding to the Dehn twist,
gives us  the characteristic function of this module.
This allows us to reduce the calculation of the perturbative quantum invariants of manifolds to
the Wilson loop invariant of links similarly  as it was done e.~g. in \cite{RT}.

We hope to elaborate all of this elsewhere.

\bibliographystyle{alphanum}
\bibliography{kon}
\end{document}